\documentclass[12pt]{amsart}
\usepackage{graphicx,color}
\usepackage{verbatim}
\usepackage{amsmath}
\usepackage{amssymb, mathrsfs}
\usepackage{amsbsy}
\usepackage{amscd}
\usepackage{amsthm}
\usepackage{bm}

\newcommand{\N}{\mathbb{N}}

\newcommand{\Z}{\mathbb{Z}}
\newcommand{\E}{\mathbb{E}}
\newcommand{\set}[1]{\left\{#1\right\}}

\renewcommand{\P}{\mathbb{P}}

\newcommand{\dd}{{\rm{d}}}

\newcommand{\RW}{P^{\omega,\mu}}

\newtheorem{theorem}{Theorem}[section]
\newtheorem{lemma}[theorem]{Lemma}

\theoremstyle{definition}
\newtheorem{remark}[theorem]{Remark}

\newtheorem{definition}[theorem]{Definition}

\numberwithin{equation}{section}
\newcounter{anc}[section]
\title[Slowdown]{Slowdown estimates for one-dimensional random walks
in random environment with %heavy tailed 
holding times}
\author{Amir Dembo}
\address{Kyoto University and Stanford University}
\email{adembo@stanford.edu}

\author{Ryoki Fukushima}
\address{Research Institute of Mathematical Sciences, Kyoto University}
\email{ryoki@kurims.kyoto-u.ac.jp}

\author{Naoki Kubota}
\address{College of Science and Technology, Nihon University}
\email{kubota.naoki08@nihon-u.ac.jp}

\begin{document}

%\date{}

\begin{abstract}
We consider a one dimensional random walk in random environment 
that is uniformly biased to one direction. In addition to the transition 
probability, the jump rate of the random walk is assumed to be spatially
inhomogeneous and random. We study the probability that the random walk 
travels slower than its typical speed and determine its decay rate asymptotic. 
\end{abstract}
% =========================================================

\maketitle
%\tableofcontents
%\pagewiselinenumbers

\section{Introduction}
\subsection{Setting and preliminaries}
Let $(\omega=\{\omega(x)\}_{x\in\Z},\mathbf{P})$ be indepnedent and identically 
distributed random variables taking values in $(0,1)$. 
For a given $\omega$, the random walk in random environment $\{X_n\}_{n=0}^\infty$ is the Markov chain with transition probability 
\begin{equation}
\label{transition}
\begin{split}
 \omega(x)&= P^\omega(X_{n+1}=x+1|X_n=x)\\
 &=1-P^\omega(X_{n+1}=x-1|X_n=x).
\end{split}
\end{equation}
It is said to be uniformly biased (to the right) if 
$\mathbf{P}\textrm{-essinf}\, \omega(0)>1/2$. In this case, the law of large 
numbers is known to hold with a positive speed (see~\cite{Sol75}):
\begin{equation}
 \lim_{n\to\infty}\frac{1}{n}X_n=v_{\mathbf{P}}>0. 
\end{equation}
In this paper, we consider a variant of this process whose jump rate is 
spatially inhomogeneous and random. Specifically, as in~\cite{DGZ04}, let
$(\mu=\{\mu(x)\}_{x \in \Z},\P)$ be independent and identically distributed strictly positive random variables of mean one. 
For a given $\mu$, we consider a continuous time random walk 
$(X=\{X_t\}_{t\ge 0},\{ P^{\omega,\mu}_z \}_{z \in \Z})$ on $\Z$ 
whose jump rates from $x$ to $x+1$ and $x-1$ are given by $\omega(x)/\mu(x)$ and
$(1-\omega(x))/\mu(x)$, respectively.
This type of random walk (usually with $\omega\equiv 1/2$) is sometimes called 
a random hopping time dynamics. If in addition the mean of $\mu(0)$ is infinite, it is also called Bouchaud's trap model.
Since we assumed finite mean, there is no trapping effect and it is easy to 
check that the law of large numbers holds:
\begin{equation}
P^{\omega,\mu}_0\left(\lim_{t\to\infty}\frac{1}{t}X_t=v_{\mathbf{P}}\right)=1\qquad
\mathbf{P}\otimes \P\textrm{-almost surely}.
\end{equation}
The large deviation principle of rate $t$ 
also holds for the law of $\{t^{-1}X_t\}_{t>0}$
as a special case of the results of ~\cite{DGZ04}.
 However, when $\mu$ is unbounded, 
it is easily seen that the \emph{slowdown} probability
\begin{equation}
P^{\omega,\mu}_0\left(X_t<vt\right)
\textrm{ for }v\in (0,v_\mathbf{P}) 
\end{equation}
exhibits sub-exponential decay. The aim of this work is to establish the 
precise sub-exponential rate of the slowdown probability decay and 
relate it to the tail of the law of $\mu(x)$. While our methods 
apply for quite general distributions
of $\mu$, we consider three representative classes (Pareto, Intermediate and 
Weibull) to make the statements concise. To be precise, suppose
\begin{equation}\label{eq:g-def}
 \P(\mu(0)>r)=\exp\{-g(r)\}
\end{equation} 
and $g$ has either of the following forms:
\begin{itemize}
 \item[(P)] $e^{g(r)}$ is regularly varying at $\infty$ with index $\alpha>1$;
 \item[(I)] $g$ is slowly varying at $\infty$ satisfying 
$\lim_{r\to\infty}g(r)/\log r=\infty$;
 \item[(W)] $g$ is regularly varying with index $\alpha>0$ at $\infty$.
\end{itemize} 
Furthermore, when $g$ satisfies (W) with $\alpha=1$, we assume in addition that the so-called Cram\'er condition holds:
\begin{itemize}
 \item[(C)] there exists $C>0$ such that $\E[e^{C\mu(0)}]<\infty$.
\end{itemize}

In the quenched slowdown estimate, the extreme value of $\mu$ plays an 
important role. Let us recall two results from extreme value theory. Let $g^{-1}$ denote the left-continuous inverse of $g$. The first result gives us a condition under which the running maxima of $\{\mu(x)\}_{x\in\Z}$ can be approximated by a deterministic sequence up to multiplicative constant.
\begin{lemma}
\label{lem:stable}
Suppose that %there exists a finite $l\ge 1$ such that for all $\epsilon>0$, 
\begin{equation}
l=\sup\{\lambda\ge 1\colon \E\left[\exp\{g(\lambda\mu(0))\}\right]<\infty\}<\infty.
% \E\left[\exp\{g((l^{-1}-\epsilon)\mu(0))\}\right]<\infty\text{ and }
% \E\left[\exp\{g((l^{-1}+\epsilon)\mu(0))\}\right]=\infty.
\label{eq:stability}
\end{equation}
Then 
\begin{equation}
l^{-1} \le \liminf_{t\to\infty}\frac{\max_{0 \le x \le t}\{\mu(x)\}}{g^{-1}(\log t)}\le 
\limsup_{t\to\infty}\frac{\max_{0 \le x \le t}\{\mu(x)\}}{g^{-1}(\log t)}\le l.
\end{equation}
\end{lemma}
This can be found in~\cite[Corollary 1]{RT73}. Any distribution in the class (W) satisfies~\eqref{eq:stability} with $l=1$. For the class (I), it is satisfied for $g(r)=(\log r)^\beta1_{\{r\ge 1\}}$ ($\beta>1$) with $l=1$ and $g(r)=\log r \log\log r 1_{\{r\ge e\}}$ with $l=e$, but not for $g(r)=\log r \log\log\log r 1_{\{r\ge e^e\}}$. No distribution in (P) satisfies~\eqref{eq:stability}: indeed, the regular variation assumption implies $\lim_{x\to\infty}g(x)-g(\delta x)=-\alpha\log \delta$, and then we know from~\cite[Lemma~2.2-(i)]{Tom86} that~\eqref{eq:stability} fails to hold for any $l\ge 1$. %\RF{(This is the sole reason why I assumed the regular variation in (P). If we only assume $g(r)=(\alpha+o(1))\log r$, then I don't know if the assumption of the above lemma always fails. And I don't want to lead readers to this pathology.)}

In order to cover the cases where~\eqref{eq:stability} fails to hold, we state another lemma which readily follows from~\cite[Theorems~3.5.1 and~3.5.2]{EKM97}:
\begin{lemma}
\label{lem:unstable}
For any $\epsilon>0$, $\P\textrm{-a.s.}$, for all sufficiently large $t$,
\begin{equation}
g^{-1}((1-\epsilon)\log t)\le \max_{0\le x\le t}\{ \mu(x) \}\le g^{-1}((1+\epsilon)\log t).
\end{equation}
\end{lemma}
This lemma aims at giving an easily computable asymptotic bound and is not sharp. Once an explicit form of the distribution is given, one can often obtain a finer asymptotics from~\cite[Theorems~3.5.1 and~3.5.2]{EKM97}. 
For example, if $g(r)=\alpha\log r 1_{\{r\ge e\}}$, then it is shown in~\cite[Lemma~3.3]{HMS08} that 
\begin{equation}
 \P\left(v_c(t) \le \max_{0\le x\le t} \{ \mu(x) \} \le u_\rho(t)
 \textrm{ eventually}\right)
=\begin{cases}
 1,&\text{if }\rho>0 \text{ and }c\in(0,1),\\
 0,&\text{if }\rho\le 0 \text{ or }c\ge 1,\\
 \end{cases}
\label{eq:maxPareto}
\end{equation}
where
\begin{align}
u_\rho(t)&=(t\log t\log\log t)^{1/\alpha}( \log\log\log t)^{1/\alpha+\rho},
\label{eq:u}\\
v_c(t)&=c(t/\log\log t)^{1/\alpha}.\label{eq:v}
\end{align}

\subsection{Results} To simplify the presentation,
%of the results, 
we introduce the following notation: 
\begin{definition} 
For two functions $f,g\colon (0,\infty)\to (0,\infty)$, we write
$f(t)\asymp_{\log} g(t)$ when there exists a $c\in(0,\infty)$ such that for all sufficiently large $t$,
\begin{equation}
 c^{-1} \log g(t)\le \log f(t) \le c \log g(t).
\end{equation}
When only the left inequality holds, we 
write $f(t) \lesssim_{\log} g(t)$. 
\end{definition}
Our first result is the following quenched slowdown estimate. 
%%%%%%%%%%%%%%%%%%%%%%%%%%%%%%%%%%%%%%%%%%%%%%%%%%%%%%%%%%%%%%%%%%%%%%%
\begin{theorem}
\label{quenched}
%\RF{(The statement is corrected.)}
Let $\mathbf{P}$ be a uniformly biased environment: 
$\mathbf{P}\textrm{-\upshape{essinf}}\,{\omega(0)}>1/2$. 
For any $v\in (0,v_\mathbf{P})$, $\mathbf{P}\otimes \P$-almost surely, 
\begin{equation}
\exp\left\{-\frac{t}{g^{-1}((1-\epsilon)\log t)}\right\}
 \lesssim_{\log}  P^{\omega,\mu}_0\left(X_t<vt\right)
 \lesssim_{\log} \exp\left\{-\frac{t}{g^{-1}((1+\epsilon)\log t)}\right\}.
\end{equation} 
Moreover, if $\mu$ satisfies the assumption of Lemma~\ref{lem:stable}, then
\begin{equation}
 P^{\omega,\mu}_0\left(X_t<vt\right)
\asymp_{\log} \exp\left\{-\frac{t}{g^{-1}(\log t)}\right\}.
\end{equation} 
\end{theorem}
%%%%%%%%%%%%%%%%%%%%%%%%%%%%%%%%%%%%%%%%%%%%%%%%%%%%%%%%%%%%%%%%%%%%%%%
Our second result is the corresponding annealed slowdown estimate. 
To this end, for each $t$, let $h(t)$ be the largest $h > 0$ satisfying
\begin{equation}
 \frac{t}{h}\ge g(h)-\log t.
\label{eq:h}
\end{equation}
Such $h$ exists for all large $t$. Indeed, 
when $t$ is large, the inequality \eqref{eq:h} 
holds for $h=\log t$ and fails for $h=t$, whereas 
per fixed $t>0$ the left hand side of \eqref{eq:h} 
is decreasing in $h$ and 
its right hand side eventually increasing in $h$.
\begin{remark}
\label{h}
In general, both $h(t)$ and $g(h(t))$ grow sub-linearly in $t$. 
Furthermore, it is straightforward to check that
\begin{equation}
h(t)\sim
\begin{cases}
\frac{t}{(\alpha-1)\log t}& \textrm{for (P)},\\[5pt]
t^\frac{1}{\alpha+1}\ell(t)& \textrm{for (W)},
\end{cases}
\end{equation}
where the function $\ell(t)$ is slowly varying at $\infty$. 
There is no such simple formula in the case (I) but for a representative 
example $g(r)=(\log r)^\beta$ ($\beta>1$), we have
\begin{equation}
 h(t)\sim \frac{t}{(\log t)^{\beta}}.
\end{equation}
\end{remark}
%%%%%%%%%%%%%%%%%%%%%%%%%%%%%%%%%%%%%%%%%%%%%%%%%%%%%%%%%%%%%%%%%%%%%%%
\begin{theorem}
\label{annealed}
Let $\mathbf{P}$ be a uniformly biased environment: 
$\mathbf{P}\textrm{-\upshape{essinf}}\,{\omega(0)}>1/2$.
Suppose $v\in (0,v_\mathbf{P})$. Then for $h(\cdot)$ of~\eqref{eq:h}, 
\begin{equation}
\mathbf{P}\otimes \P\left[ P^{\omega,\mu}_0\left(X_t<vt\right)\right]
\asymp_{\log}\exp\left\{-\frac{t}{h(t)}\right\},
\end{equation}
where $\mathbf{P}\otimes \P[\,\cdot\,]$ denotes the expectation with respect to $\mathbf{P}\otimes \P$.
\end{theorem}
%\begin{remark}
%\RF{(I think we don't need this remark anymore. The assumption for the distribution of $\mu$ is now fairly general.)}
%As is mentioned above, the assumption~\eqref{eq:g-def} with (P), (I) or (W) is rather restrictive. 
%However, it is easy to adapt our arguments for establishing somewhat 
%weaker results in a more general setting. For example, if we only assume 
%\begin{equation}
%\P(\mu(0)>r)=r^{-\alpha+o(1)} \textrm{ as }r\to\infty,
%\end{equation}
%then we have 
%\begin{equation}
%P^{\omega,\mu}_0\left(X_t<vt\right)
%=\exp\left\{-t^{\frac{\alpha-1}{\alpha}+o(1)}\right\}
%\textrm{ as }t\to\infty
%\end{equation}
%$\mathbf{P}$-almost surely, as well as
%\begin{equation}
% \mathbf{P}\otimes \P\left[ P^{\omega,\mu}_0\left(X_t<vt\right)\right]
%=t^{-\alpha+1+o(1)} \textrm{ as }t \to\infty.
%\end{equation}
%\end{remark}
%%%%%%%%%%%%%%%%%%%%%%%%%%%%%%%%%%%%%%%%%%%%%%%%%%%%%%%%%%%%%%%%%%%%%%%
\subsection{Related works}
%The results in this work covers \emph{uniformly biased} random walk in random environment with spatially inhomogeneous holding times. 
Sub-exponential tail estimates are established in~\cite{DPZ96}
for the annealed slowdown probability of
the random walk in random environment $(X_n)_{n \ge 0}$ governed by 
\eqref{transition}. 
When the environment assumes both positive and negative drift, that is, 
$\mathbf{P}\textrm{-\upshape{essinf}}\,\omega(0) <1/2<
\mathbf{P}\textrm{-\upshape{esssup}}\,\omega(0)$, the annealed slowdown probability exhibits a polynomial decay. 
%In particular, the holding time can cause a visible effect only if the distribution of $\mu$ has power law tail. In fact, it seems that the results as precise as in~\cite{DPZ96} can be proved with a minor modification of the argument thereof. Then one can try to prove the quenched result following a block argument as in~\cite{GZ98}. 
In the case of positive and zero drift, that is, 
$\mathbf{P}\textrm{-\upshape{essinf}}\,\omega(0) =1/2$, under an additional assumption that $\mathbf{P}(\omega(0)=1/2)>0$, the slowdown probability is shown to decay stretched exponentially with exponent $1/3$. 
The rate of decay of the corresponding quenched slowdown probability 
is determined in~\cite{GZ98}, based on the annealed result and the block argument. In the case of positive and zero drift, both annealed and quenched results are refined to the precision of the usual large deviation principle in~\cite{PPZ99} and~\cite{PP99}, respectively. On the other hand, in the case of positive and negative drift, it has recently been shown in~\cite{AP16} that the leading order of the quenched slowdown probability oscillates and hence does not satisfy a large deviation principle. For more details, we refer the reader to a survey article~\cite{GZ99} as well as the introduction of~\cite{AP16}. 

We focus here on the uniformly biased situation with (inhomogeneous) holding times. Indeed, for uniformly biased environments, the result of~\cite{GdH94}
shows that without holding times the slowdown probability decays exponentially. Thus, the sub-exponential decay of the slowdown probability is caused in this setting
solely by the inhomogeneity of holding times. 
%It is interesting to study the slowdown estimates in the other cases, with holding times. 
Note that in case of positive and negative drift, since the annealed slowdown probability decays polynomially without holding times, the holding time can cause a visible effect only if $\mu$ has a power law tail. Similarly, in the case of positive and zero drift, the most natural choice of $\mu$ would be the Weibull distribution. We leave the latter two cases for future research. 

Finally, \cite{Szn99,Szn00,Ber12} provide estimates for the 
decay rate of the slowdown probability for random walk 
in random environment in higher dimensions. While it is interesting to see how the holding times affect such slowdown probabilities, 
our method relies on a certain renewal structure which is 
limited to the one dimensional setting.
 
%%%%%%%%%%%%%%%%%%%%%%%%%%%%%%%%%%%%%%%%%%%%%%%%%%%%%%%%%%%%%%%%%%%%%%%
\subsection{Outline}
Section~\ref{S:LB} provide the relatively easy proofs of our lower bounds,
where in both
quenched and annealed settings, we simply let the random walk stay until time $t$ at the site of the highest $\mu$-value within $[0,vt-1]$. In the quenched case, the highest $\mu$-value behaves as in Lemmas~\ref{lem:stable} and~\ref{lem:unstable}, 
while 
in the annealed setting, we can make it larger at a suitable cost in $\log \P$-probability, so we optimize the sum of the corresponding 
cost and gain. 

The derivation of the upper bounds is more involved. Since a 
sub-exponential slowdown decay for a
uniformly biased random walk can only be caused by the inhomogeneity 
of the holding times $\mu$, we introduce in Section~\ref{S:prelim}
a suitable time change and thereby reduce such upper bounds on the slowdown decay to a tail bound for certain additive functionals. Section~\ref{S:onGood} 
provides our main technical contribution, showing that conditioning on 
some good events with respect to $\omega$ and $\mu$ yields the stated 
upper bounds. 
Finally, in Section~\ref{sec-UB}, we show that
\begin{enumerate}
 \item in the quenched setting, the good event has probability one; 
 \item in the annealed setting, the good event has probability comparable to the upper bound in Theorem~\ref{annealed}. 
\end{enumerate}

\section{Lower bounds}\label{S:LB}
\begin{proof}
[Proof of the lower bound in Theorem~\ref{quenched}]
Let us define a regularly varying function
 
\begin{equation}
 M(t)=
 \begin{cases}
 g^{-1}((1-\epsilon)\log t),&\textrm{or}\\
 g^{-1}(\log t),&\textrm{when \eqref{eq:stability} holds.} 
 \end{cases}
 \end{equation}
Then due to Lemmas~\ref{lem:stable} and~\ref{lem:unstable}, the following holds $\P$-almost surely:
for all sufficiently large $t$, there exist a point $x\in[0,vt-1]$ and $c_v>0$ (which is independent of $x$) such that
$\mu(x)\ge c_vM(t)$. It follows that 
\begin{equation}
 \RW_0(X_t < vt)\ge \RW_0(\tau_1(x)\ge t)
 \ge \exp\set{-\frac{t}{c_vM(t)}},
\end{equation}
where $\tau_1(x)$ is the first holding time at $x$, which is distributed
as $\textrm{Exp}(1/\mu(x))$. 
\end{proof}
\begin{proof}
[Proof of the lower bound in Theorem~\ref{annealed}]
As in the quenched case, if there is a point 
$x\in[0,vt-1]$ such that $\mu(x)\ge h(t)$, then we can slow
the random walk by using the first holding time at $x$. 
Therefore, by the definition of $h(t)$, we have 
\begin{equation}
\begin{split}
\mathbf{P}\otimes \P\left[ P^{\omega,\mu}_0(X_t < vt) \right]
&\ge \exp\set{-\frac{t}{h(t)}}
 \P\left(\max_{x\in[0,vt-1]}\mu(x)\ge h(t) \right)\\
&\asymp_{\log}\exp\set{-\frac{t}{h(t)}+\log t-g(h(t))}\\
&\ge \exp\set{-\frac{2t}{h(t)}},
\end{split}
\end{equation}
which is the desired lower bound.
\end{proof}
%%%%%%%%%%%%%%%%%%%%%%%%%%%%%%%%%%%%%%%%%%%%%%%%%%%%%%%%%%%%%%%%%%%%%%%

\section{Preliminaries for upper bounds}
\label{S:prelim}
\subsection{Reduction to tail estimate for additive functional}
\label{additive}
Let us first translate the problem in terms of the hitting time $H(x)$ of 
$x$ by our process.
For any $u>v$, on the slowdown event $\{X_t<vt\}$, either the walk hit 
$ut$ before time $t$ and thereafter go back to $(-\infty,vt]$ or 
it does not reach $ut$ before time $t$. Hence,
\begin{equation}
\label{hitting}
\begin{split}
P^{\omega,\mu}_0(X_t<vt) 
\le P^{\omega,\mu}_0(H(ut)< t, X_t<vt)+P^{\omega,\mu}_0(H(ut)\ge t). 
\end{split}
\end{equation}
The first term on the right hand side is exponentially small in $t$. Indeed,
since the random walk then must backtrack for length $(u-v)t$, 
it follows that 
\begin{equation}
\begin{split}
P^{\omega,\mu}_0(H(ut)< t, X_t<vt)
& \le \RW_{ut}(H(vt)<\infty)\\
& = P^\omega_{ut}(H(vt)<\infty) \le \exp\{-c(u-v)t\}
\end{split}
\end{equation} 
(see Lemma~\ref{exp-decay} for the last inequality). To handle the other
term on the right side of \eqref{hitting} we utilize the following time 
change description of $P^{\omega,\mu}$. 
Let $(\{S_t\}_{t\ge 0}, \{P_x^\omega\}_{x\in\Z})$ %\RF{(Referee suggested $P_0^\omega$ but I think this is more consistent.)} 
be the continuous time random walk in
random environment, that is, it jumps after $\rm{Exp}(1)$ time regardless of
its position and where it moves obeys the rule~\eqref{transition}. Then,
define the strictly increasing positive continuous additive functional
\begin{equation}\label{eq:amir1}
 A_\mu(t)=\int_0^t \mu(S_r)\dd r
\end{equation}
and denote by $A^{-1}_\mu$ its inverse. 
Since the process $\{S_{A_\mu^{-1}(t)}\}_{t\ge 0}$ 
has under $P_0^\omega$ the same law as $\{X_t\}_{t\ge 0}$ has 
under $P^{\omega,\mu}_0$, the second term in~\eqref{hitting} is precisely 
\begin{equation}
 P_0^\omega(H(ut)\ge A_\mu^{-1}(t))
= P_0^\omega(A_\mu(H(ut))\ge t).
\end{equation}
%where $H^S(ut)$ is the first passage time through $ut$ for the random walk $\{ S_t \}_{t \geq 0}$.

\subsection{Green function estimates}
The Green function for $\{S_t\}_{t \geq 0}$ is defined 
for $-\infty \leq a<x<b \leq \infty$ as 
\begin{equation}
G_{(a,b)}^\omega(x,y)
=E_x^\omega\left[\int_0^{H(a)\wedge H(b)}1_{\{y\}}(S_r)\dd r\right].
\end{equation}
Since our random walk is transient, this quantity
is finite 
$\mathbf{P}$-almost surely. 
\begin{lemma}
\label{exp-decay}
There exist positive constants $c_1$, $c_2$ 
and a function $\eta(\epsilon)\to 0$ as $\epsilon\to 0$ that depend only on 
$\mathbf{P} \textrm{-essinf }\,\omega(0)>1/2$ 
such that the following hold $\mathbf{P}$-almost surely:
\begin{enumerate}
\item for any $z\in\Z$, 
$E^\omega_{z-1}[\exp\{\epsilon H(z)\}]\le 1+\eta(\epsilon)$, 
%where $o_\epsilon(1)$ denotes a quantity that vanishes as 
%$\epsilon\downarrow 0$ uniformly in $\omega$ and $z$;
\item for all $x \geq y$,
 both $P^{\omega}_x(H(y)<\infty)$ and $G_{(-\infty,\infty)}^\omega(x,y)$
 are bounded by $c_1\exp\{-c_2(x-y)\}$.
\end{enumerate}
\end{lemma}
\begin{proof}
Let $(\{\underline{S}_t\}_{t\ge 0},\{ P_x^{\underline{\omega}}\}_{x \in \Z})$ be the biased random walk corresponding to the deterministic 
environment $\omega\equiv \mathbf{P}\textrm{-essinf }\,\omega(0)>1/2$. 
It is standard to construct a coupling $(S_t,\underline{S}_t)$ so that the 
two walks jump at the same time and $\underline{S}_t\le S_t$ for all $t\ge 0$. 
The first assertion (i) readily follows from this coupling since 
for any $\epsilon >0$,
\begin{equation}
\begin{split}
 \sup_{z\in \Z}  E^\omega_{z-1}[\exp\{\epsilon H(z)\}]
&\le 
 E^{\underline{\omega}}_{0}[\exp\{\epsilon H(1)\}]
=: 1+\eta(\epsilon).
\end{split}
\end{equation} 
Next, turning to the proof of (ii), by our coupling
$P_z^\omega(H(z-1)=\infty) \ge \eta$, 
uniformly in $\omega$ and $z\in\Z^d$, where
$\eta := P_z^{\underline{\omega}}(H(z-1)=\infty)$ 
is positive. Hence, by the strong Markov property, 
\begin{equation}\label{H-bd}
P^{\omega}_x(H(y)<\infty)
=\prod_{z=x}^{y-1}P_z^\omega(H(z-1)<\infty) \le (1-\eta)^{y-x}
\end{equation}
as claimed. 
Another application of the strong Markov property yields the identity
\begin{equation}\label{G-ident}
G_{(-\infty,\infty)}^\omega(x,y)
=P^{\omega}_x(H(y)<\infty)\, G_{(-\infty,\infty)}^\omega(y,y).
\end{equation}
Further, by our coupling with $\{P^{\underline{\omega}}_x\}_{x\in\Z}$,
the return probability of the process $(S_t)_{t \ge 0}$ 
is bounded away from one, uniformly in its starting point $S_0=y$ 
and the environment $\omega$.
This implies same uniform boundedness of $G_{(-\infty,\infty)}^\omega(y,y)$, 
which in view of \eqref{H-bd} and \eqref{G-ident} completes the proof.
\end{proof}
%%%%%%%%%%%%%%%%%%%%%%%%%%%%%%%%%%%%%%%%%%%%%%%%%%%%%%%%%%%%%%%%%%%%
\section{Upper bound on a good event}
\label{S:onGood}
Let us fix $u\in(v,v_{\mathbf{P}})$ and a regularly varying and increasing function 
$M$ with $\lim_{t\to\infty}M(t)=\infty$. Throughout this section, we will fix $\omega$ and $\mu$ and assume that they satisfy the following conditions:
\begin{equation}
\frac{1}{M(ut)}\max_{-\epsilon t\le z\le ut}\{\mu(z)\}\le 1,
\label{cond:max}
\end{equation}
and there exists $\delta\in(0,1)$ such that
\begin{equation}
\sum_{y\in (-\epsilon t, ut)} c^\omega_{-\epsilon t}(y) \mu(y) 
\le (1-\delta)t
\label{cond:w-sum}
\end{equation}
holds for 
\begin{equation}
 c^\omega_{-\epsilon t} (y) := \sum_{x\in (y,ut]} G_{(-\epsilon t,x)}^{\omega} (x-1,y) 
\end{equation}
and all sufficiently small $\epsilon>0$.
Let us introduce 
\begin{equation}\label{eq:feps-def}
f_{\epsilon,t}(x,y)=
E_x^\omega\left[\exp\set{\frac{\epsilon }{M(ut)}
\int_0^{H(y)\wedge H(-\epsilon t)}
\mu(S_r)\dd r}\right]
\end{equation}
for $0\le x<y\le ut$. 
By the strong Markov property and the fact $f_{\epsilon,t}(x,y)\ge 1$, this can be shown to be sub-multiplicative in the following sense: for any $0\le x< y < z \le ut$, 
\begin{equation}
f_{\epsilon,t}(x,z) \le f_{\epsilon,t}(x,y)f_{\epsilon,t}(y,z). 
\label{eq:sub-mult}
\end{equation}
%(This part is related to referee's comment to~\eqref{gen-fcn}. A proof is given in \TeX\ file in case you want to see.)
\if0
\begin{equation}
\begin{split}
E_x^\omega&\left[\exp\set{\frac{\epsilon }{M(ut)}
\int_0^{H(z)\wedge H(-\epsilon t)}\mu(S_r)\dd r}\right]\\
&= E_x^\omega\left[\exp\set{\frac{\epsilon }{M(ut)}
\int_0^{H(z)\wedge H(-\epsilon t)}\mu(S_r)\dd r}\colon H(-\epsilon t)>H(y)\right]\\
&\quad+E_x^\omega\left[\exp\set{\frac{\epsilon }{M(ut)}
\int_0^{H(-\epsilon t)}\mu(S_r)\dd r}\colon H(-\epsilon t)\le H(y)\right]\\
&=E_x^\omega\left[\exp\set{\frac{\epsilon }{M(ut)}
\int_0^{H(y)}\mu(S_r)\dd r} 
f_{\epsilon,t}(y,z) \colon H(-\epsilon t)>H(y)\right]\\
&\quad+E_x^\omega\left[\exp\set{\frac{\epsilon }{M(ut)}
\int_0^{H(-\epsilon t)}\mu(S_r)\dd r}\colon H(-\epsilon t)\le H(y)\right]\\
&\le E_x^\omega\left[\exp\set{\frac{\epsilon }{M(ut)}\int_0^{H(y)\wedge H(-\epsilon t)}\mu(S_r)\dd r} \right]
f_{\epsilon,t}(y,z), 
\end{split}
\end{equation}
where in the last line, we have used that $f_{\epsilon,t}(y,z)\ge 1$.
\fi
The assumption~\eqref{cond:max} and Lemma~\ref{exp-decay} imply
\begin{equation}
\begin{split}
f_{\epsilon,t}(x,y)
& \le E_x^\omega\left[\exp\set{\epsilon H(y)}\right]
=\prod_{z=x+1}^yE_{z-1}^\omega\left[\exp\set{\epsilon H(z)}\right]
\leq (1+\eta(\epsilon))^{y-x}.
\end{split}
\end{equation}
By the Feynman--Kac formula (Theorem~6.7 in~\cite{Dem05}), we have 
\begin{equation}
\begin{split}
f_{\epsilon,t}(x-1,x)
&= 1+ \frac{\epsilon }{M(ut)}
 \sum_{y\in (-\epsilon t,x)} G^\omega_{(-\epsilon t,x)}(x-1,y)
\mu(y) f_{\epsilon,t}(y,x)\\
&\le 1+ \frac{\epsilon }{M(ut)} 
 \sum_{y\in (-\epsilon t,x)} G^\omega_{(-\epsilon t,x)}(x-1,y)\mu(y) 
 (1+\eta(\epsilon))^{x-y}.
\end{split}
\label{FK}
\end{equation}
Now using $\log (1+x)\le x$ for $x\ge 0$, we obtain from~\eqref{eq:sub-mult} and \eqref{FK} that 
%\RF{(In the first line below, I think referee is right. This is due to the extra $\wedge H(-\epsilon t)$.)}
\begin{equation}
\begin{split}
\log f_{\epsilon,t}(0,ut)  
& \le \sum_{1 \leq x \leq ut}\log f_{\epsilon,t}(x-1,x)\\
& \le \sum_{1 \leq x \leq ut}
% \log \left(1+ 
\frac{\epsilon}{M(ut)} 
 \sum_{y\in (-\epsilon t,x)} G^\omega_{(-\epsilon t,x)}(x-1,y)\mu(y) 
(1+\eta(\epsilon))^{x-y}
% \right)
\\
&\le \frac{\epsilon }{M(ut)}
 \sum_{y\in (-\epsilon t,ut)}\mu(y) 
\sum_{x \in (y,ut]}G^\omega_{(-\epsilon t,x)}(x-1,y)(1+\eta(\epsilon))^{x-y}.
\end{split}
\label{gen-fcn}
\end{equation}
%We next show that the above sum in $x\in (y,ut]$ is well approximated by $c^\omega_{-\epsilon t} (y)$. 
Next, by Lemma~\ref{exp-decay}(ii) we have that 
uniformly in $\omega$, $t$ and $y$,
as $\epsilon\to 0$, 
\begin{equation}
\begin{split}
0 & \le \sum_{x \in (y,ut]}G^\omega_{(-\epsilon t,x)}(x-1,y)(1+\eta(\epsilon))^{x-y}
-c^\omega_{-\epsilon t} (y)\\
%&\quad=\sum_{x \in (y,ut]}G^\omega_{(-\epsilon t,x)}(x-1,y)
%((1+\eta(\epsilon))^{x-y}-1)\\
& \le \sum_{x \in (y,\infty)}c_1e^{-c_2(x-y)}((1+\eta(\epsilon))^{x-y}-1)
\to 0.
\end{split}
\label{rf1}
\end{equation}
%, by the dominated convergence theorem. 
Further, $c^\omega_{-\epsilon t} (y)\ge G_{(-\epsilon t,y+1)}^{\omega} (y,y) \ge 1$  (the latter being the expected first jump time of 
$\{S_t\}$ out of $y$). Consequently, it follows from~\eqref{rf1} that 
\begin{equation}
 \sum_{x \in (y,ut]}G^\omega_{(-\epsilon t,x)}(x-1,y)(1+\eta(\epsilon))^{x-y} \le (1+\delta)c^\omega_{-\epsilon t} (y)
\end{equation}
for all sufficiently small $\epsilon>0$, uniformly in $y$. Substituting this to~\eqref{gen-fcn}, we deduce thanks to~\eqref{cond:w-sum} that  
\begin{equation}
\log f_{\epsilon,t}(0,ut)  
\le
 \frac{\epsilon (1+\delta)}{M(ut)}
 \sum_{y\in (-\epsilon t,ut)}\mu(y)
c^\omega_{-\epsilon t} (y) \le \frac{\epsilon(1-\delta^2)t}{M(ut)}.
\label{gen-fcn2} 
\end{equation}
Thus, using Chebyshev's inequality and recalling \eqref{eq:amir1} and 
\eqref{eq:feps-def}, we obtain that  
\begin{equation}
\begin{split}
%&
P_0^\omega\left(A_{\mu}(H(ut)\wedge H(-\epsilon t))\ge t\right)
%\\
&
%\quad
\le \exp\set{-\frac{\epsilon t}{M(ut)}}
f_{\epsilon,t}(0,ut) 
% E_0^\omega\left[\exp\set{\frac{\epsilon}{M(ut)}
% \int_0^{H(ut)\wedge H(-\epsilon t)}\mu(S_u)\dd u}\right]
\\
&
%\quad
\le \exp\set{-\frac{\epsilon \delta^2 t}{M(ut)}}.
\end{split}
\end{equation}
Recall Lemma~\ref{exp-decay}(ii) that 
$P^\omega_0(H(-\epsilon t)<\infty)$ decays exponentially in $t$, 
so we can choose $\epsilon>0$ such that 
for all sufficiently large $t$,
\begin{equation}
\begin{split}
P^\omega_0\left(A_{\mu}(H(ut))\ge t\right)
&\le P^\omega_0\left(A_{\mu}(H(ut) \wedge H(-\epsilon t))\ge t\right)
+\exp\{-c_\epsilon t\}\\
&\le 2\exp\set{-\frac{\epsilon \delta^2 t}{M(ut)}}.
\end{split}
\end{equation}
Referring to Subsection~\ref{additive}, this leads us to an upper bound
\begin{equation}
 P^{\omega,\mu}_0(X_t<vt)\lesssim_{\log}\exp\set{-\frac{t}{M(t)}}
\label{UB}
\end{equation} 
under the assumptions~\eqref{cond:max} and~\eqref{cond:w-sum}
(where we also used the fact that $t \mapsto M(t)$
is regularly varying at infinity).
%%%%%%%%%%%%%%%%%%%%%%%%%%%%%%%%%%%%%%%%%%%%%%%%%%%%%%%%%%%%%%%%%%%%
\section{Proofs of upper bounds}
\label{sec-UB}
\begin{proof}
[Proof of the upper bound in Theorem~\ref{quenched}]
In view of~\eqref{UB}, we have only to show that~\eqref{cond:max} 
and~\eqref{cond:w-sum} are satisfied for a suitable $M$ and 
$u\in(v,v_\mathbf{P})$. 
Thanks to Lemmas~\ref{lem:stable} and~\ref{lem:unstable}, if we choose 
\begin{equation}
M(t)=
\begin{cases}
 g^{-1}((1+\epsilon)\log t),&\textrm{or}\\
 u_\rho(t),&\textrm{when~\eqref{eq:stability} holds,} 
\end{cases}
\end{equation}
with $\epsilon>0$, then~\eqref{cond:max} holds $\P$-almost 
surely for all sufficiently large $t$. 

Let us turn to verify the second condition \eqref{cond:w-sum}. 
Replacing $\mu(y)$ by their mean value $\E[\mu(y)]=1$,
on the left side of \eqref{cond:w-sum}, yields 
\begin{equation}
\begin{split}
  \sum_{y\in (-\epsilon t,ut)} c^\omega_{-\epsilon t}(y) 
 % \sum_{x\in (y,ut)} 
 % G_{(-\epsilon t,x)}^{\omega} (x-1,y)
 &\le \sum_{x\in (-\epsilon t,ut]} 
 \sum_{y\in (-\infty, x)} 
 G_{(-\infty,x)}^{\omega} (x-1,y)\\
 &=E^\omega_{-\epsilon t}[H(ut)].
\end{split}
\label{total-time}
\end{equation}
The last expression is $\mathbf{P}$-almost surely of size $\frac{u+\epsilon}{v_{\mathbf{P}}} (t+o(t))$ as $t\to\infty$. 
This can be seen as follows: in the same way as in~\cite[(1.16)]{Sol75}, we find that for $\mathbf{P}$-a.e.~$\omega$, $t^{-1}{H(ut)}\to{(u+\epsilon)}/{v_{\mathbf{P}}}$ in $P_{-\epsilon t}^{\omega}$-probability as $t\to\infty$. 
On the other hand, by the same coupling as in the proof of Lemma~\ref{exp-decay}, it follows that $\sup_{t,\omega} E^\omega_{-\epsilon t}[(H(ut) /t)^2]<\infty$. 
This implies the uniformly integrablity of $\{H(ut)/t\}_{n\in\N}$ with respect to $P^{\omega}_{-\epsilon t}$ for every $\omega$, and hence the above in probability convergence can be upgraded to the convergence in $L^1$. 

For any fixed $u<v_{\mathbf{P}}$, we can choose $\delta>0$ such that 
$\frac{u+\epsilon}{v_{\mathbf{P}}}<1-\delta$ for all sufficiently small 
$\epsilon$, so it remains to control the discrepancy on the left-side 
of \eqref{cond:w-sum} due to replacing 
$\mu(y)$ by $\E[\mu(y)]=1$. 
To this end, note that from Lemma~\ref{exp-decay}-(ii) we have that
the weights $c^\omega_{-\epsilon t}(y)$ are uniformly bounded. 
Thus, applying the strong law of large numbers for the weighted sum of 
zero mean i.i.d.~variables $\{\mu(y)-1\}$, with such weights 
$\{c^\omega_{-\epsilon t}(y)\}$ yields that 
$\P\otimes \mathbf{P}$-almost surely, 
\begin{equation}
\sum_{y\in (-\epsilon t,ut)} c^\omega_{-\epsilon t}(y) (\mu(y)-1)
=o(t)
\end{equation}
as $t\to\infty$ and we are done.
\end{proof}

In order to prove the upper bound in Theorem~\ref{annealed}, we need the following lemma, which states that~\eqref{eq:h} is not too far from an equality. 
\begin{lemma}
Let $h(t)$ be as in~\eqref{eq:h}. Then for sufficiently large $t$, 
\begin{equation}
\frac{t}{h(t)}\le 2(g(h(t))-\log t). 
\label{eq:h/2}
\end{equation}
\end{lemma}
\begin{proof}
We claim that for some $c<\infty$ and all $t$ large enough
$h(t)\le ct/\log t$. Indeed, for 
$h(t)> ct/\log t$ the left side of~\eqref{eq:h} is 
smaller than $c^{-1}\log t$. On the other hand, since eventually
$g(h)\ge \beta \log h$ for some $\beta>1$ in all three cases (P), (I) and (W), the right hand side of~\eqref{eq:h} must then be at least 
$\frac{1}{2} (\beta-1)\log t$ for all large $t$. Thus, by
\eqref{eq:h} we must have 
$h(t) \le ct/\log t$ for 
$c = 2/(\beta-1)$ and all $t$ large enough.

Now by the definition of $h(t)$, it follows that for any $\lambda>1$, 
\begin{equation}
 \lambda h(t)(g(\lambda h(t))-\log t)>t,
\end{equation}
which implies 
\begin{equation}
\begin{split}
h(t)(g(h(t))-\log t)&> \frac{g(h(t))}{g(\lambda h(t))}\frac{t}{\lambda}
+h(t)\log t\left(\frac{g(h(t))}{g(\lambda h(t))}-1\right)\\
&\ge\frac{g(h(t))}{g(\lambda h(t))}\frac{t}{\lambda}
 +ct\left(\frac{g(h(t))}{g(\lambda h(t))}-1\right),
\end{split}
\end{equation}
where in the second line, we have used that $g(\cdot)$ 
is increasing and $h(t)\le ct/\log t$. The stated conclusion~\eqref{eq:h/2} 
thus holds whenever 
\begin{equation}
 \lim_{\lambda \downarrow 1} %\Big\{ 
\frac{g(h(t))}{g(\lambda h(t))} %\Big\} 
> \frac{c+\frac{1}{2}}{c+1}.
\end{equation}
To complete the proof, recall that for increasing and regularly varying $g(\cdot)$
the left hand side gets arbitrarily close to one as $h(t) \to \infty$. 
\end{proof}

\begin{proof}
[Proof of the upper bound in Theorem~\ref{annealed}]
Again in view of~\eqref{UB} and~\eqref{eq:h/2}, it remains to show that for any fixed 
$u\in(v,v_\mathbf{P})$ and small $\delta,\epsilon>0$ such that 
$\frac{u+\epsilon}{v_{\mathbf{P}}}<1-2\delta$, one has 
\begin{align}
 \P\left(\max_{-\epsilon t\le z\le ut}\{\mu(z)\}>h(t)\right)
& \lesssim_{\log}\exp\{-g(h(t))+\log t\},
\label{prob:max}\\
 \P\otimes\mathbf{P}
\left(\sum_{y\in (-\epsilon t, ut)} c^\omega_{-\epsilon t}(y)\mu(y) 
> (1-\delta)t\right)
& \lesssim_{\log}\exp\{-g(h(t))+\log t\}.
\label{prob:w-sum}
\end{align}
The bound \eqref{prob:max} follows by the definition of $g(\cdot)$ and the 
union bound. Turning to \eqref{prob:w-sum}, recall \eqref{total-time}
that the event on its left-side is contained in the union of:
\begin{align}\label{event1}
E^\omega_{-\epsilon t}[H(ut)] 
% \ge \sum_{y\in (-\epsilon t, ut)} c^\omega_{-\epsilon t}(y) 
& > (1-2\delta)t,\\
\sum_{y\in (-\epsilon t, ut)} c^\omega_{-\epsilon t}(y) (\mu(y)-1) 
& > \delta t.
\label{event2}
\end{align}
To bound the probability of \eqref{event1} we
use Jensen's inequality to find that 
\begin{equation}
\begin{split}
\mathbf{P}\left(E^\omega_{-\epsilon t}[H(ut)]>(1-2\delta)t\right)
&\le \inf_\lambda e^{-\lambda(1-2\delta)t}\mathbf{E}\left[
\exp\set{\lambda E^\omega_{-\epsilon t}[H(ut)]}\right]\\
&\le \inf_\lambda e^{-\lambda(1-2\delta)t}
\mathbf{E}\left[E^\omega_{-\epsilon t}[\exp\{\lambda H(ut)\}]\right].
\end{split}
\end{equation}
The last expression is precisely the large deviation upper bound for the 
hitting time, which is shown in~\cite{CGZ00} to decay exponentially in 
$t$ whenever $1-2\delta>(u+\epsilon)/v_\mathbf{P}$. Recall Remark~\ref{h}
that $t \mapsto g(h(t))$ grows sub-linearly, hence any event having 
exponential decay in $t$ is negligible for the purpose of verifying 
\eqref{prob:w-sum}. 
Turning to similarly control the probability of the event in 
\eqref{event2} recall that 
the positive $c^\omega_{-\epsilon t}(y)$ are
bounded away from zero and infinity, uniformly in $\omega$, 
$t$ and $y$. Thus, standard large deviation estimates for such
weighted sums yield that 
\begin{equation}\label{eq:amir2}
 \P\left(\sum_{y\in (-\epsilon t,ut)} c^\omega_{-\epsilon t}(y)
 (\mu(y)-1) > \delta t\right) \lesssim_{\log}\exp\{-g(h(t))+\log t\} 
\end{equation}
as claimed. Indeed, if $\mu(y)$ has a finite exponential moment
(which we have for (W), when $\alpha \ge 1$, see (C) in case $\alpha=1$),
then the Chernoff bound yields an exponential decay in $t$ of the 
probability on the left-side, whereas appealing
to Remark~\ref{h} for the sub-linear growth of $t \mapsto h(t)$, 
in case $\mu(y)$ has no finite exponential moments, we get 
\eqref{eq:amir2} as a special case of Lemma~\ref{tail} below.
\end{proof}

\begin{lemma}
\label{tail}
Let $(\{ \mu_k \}_{k \in \N},\P)$ be a family of i.i.d.~mean-one
random variables obeying either {\upshape (P)}, {\upshape (I)} or {\upshape (W)} with $\alpha <1$. 
Then, for any sequence $\{w_k\}_{k \in \N} \subset [0,\kappa]$ 
and $\delta>0$, there exists $c<\infty$ depending only on
$\delta$, $\kappa<\infty$ and 
$\alpha$ (which appears in conditions~{\upshape (P)} and {\upshape (W)}), 
such that 
\[
\P\left(\sum_{k=1}^nw_k(\mu_k-1)>\delta n\right) 
\le
\begin{cases}
 c n^{1-\alpha},
&\textrm{for {\upshape (P)}},\\[5pt]
\exp\set{-c^{-1} g(\delta n)},&\textrm{for {\upshape (I)} and {\upshape (W)} with $\alpha<1$.}
\end{cases}
\]
\end{lemma}
Such behavior of the large deviation estimates 
for sums of independent random variables is well-known 
in the literature. However, we were not able to find results in this specific form, hence for reader's convenience 
include its proof in the appendix.

\appendix
\section{Proof of Lemma~\ref{tail}}
%When $\mu$ obeys (W) with $\alpha\ge 1$ \AD{and has finite exponential moments
%(see (C)),} the stated estimate can be proved in the same way as the Cram\'er theorem.

For case (P), suppose that the independent
$\nu_k=w_k(\mu_k-1)$ are such that
\begin{align}
 A_t^+&=\sum_{k=1}^n\E[\nu_k^t\colon \nu_k\ge 0] < \infty
% A_t(y)&=\sum_{k=1}^n\E[|\nu_k|^t\colon -y_k\le \nu_k\le 0],\\
% B(y)&=\sum_{k=1}^n\E[\nu_k^2\colon \mu_k\le y_k].\\
% M(y)&=\sum_{k=1}^n\E[\mu_k\colon \mu_k\le y_k].
\end{align}
for some $t\in[1,2]$. Then, by \cite[Corollary 1.6]{Nag79} we have 
that for any $y\in [(4A_t^+)^{1/t},x)$,
\begin{equation}\label{eq:amir3}
\P\left(\sum_{k=1}^n \nu_k>x\right)
 \le \sum_{k=1}^n\P(\nu_k>y)+\left(\frac{e^2A_t^+}{xy^{t-1}}\right)^{x/2y}.
\end{equation}
With $c := \E[\mu_k^t]<\infty$ for $t=\min\{2,(\alpha+1)/2\}>1$, 
it follows that $A_t^+\le c\kappa^t n$.
Thus, fixing $0<\epsilon<\delta$, from \eqref{eq:amir3}
with $x=\delta n$ and 
$y=\epsilon n$, we obtain that
 \begin{equation}\label{eq:rf-2}
\P\left(\sum_{k=1}^n \nu_k> \delta n\right)
\le n\P\left(\mu_1-1 \ge \frac{\epsilon n}{\kappa} \right)
+\left(\frac{e^2 c \kappa^t}{\delta \epsilon^{t-1} n^{t-1}}\right)^{\delta/2\epsilon}
\end{equation}
as soon as $n^{1-1/t} \ge \epsilon^{-1} (4 c \kappa^t)^{1/t}$.
The first term on the right hand side of \eqref{eq:rf-2}
has the desired form while the second term there is
negligible when $\delta (t-1) > 2\epsilon (\alpha -1)$. 

The case (W) with $\alpha<1$ and (I) are studied in~\cite{Nag69a,Nag69b} and~\cite{Roz93}, respectively, for the i.i.d.~setting.
Utilizing a standard truncation argument, we extend their results to 
our weighted case in the large deviation regime. Specifically, note first that for any $0<\epsilon<\delta$, 
\begin{equation}
\begin{split}
\P\Big(\sum_{k=1}^n \nu_k>\delta n\Big)
& \le \P\left(\max_{1\le k \le n} \{\mu_k\}> \epsilon n\right)
+\P\Big(\sum_{k=1}^n \nu_k>\delta n, 
\max_{1\le k \le n} \{\mu_k\} \le \epsilon n\Big) \\
& \le n e^{-g(\epsilon n)} \quad
+ \quad 
e^{-\delta g(n)}\prod_{k=1}^n
\E\left[\exp\left\{\frac{g(n)}{n}\nu_k\right\}
\colon \mu_k\le \epsilon n\right].
\end{split}\label{eq:rf3}
\end{equation}
The first term has the desired form since $g(\cdot)$ is regularly varying 
% including index 0 in (I) 
and grows faster than the logarithm. 
It thus suffices to show that the product term on the right of \eqref{eq:rf3}
is bounded by $\exp\{\epsilon g(n)\}$. To this end, recalling that  
$e^x \le 1 + x + x^2 e^\kappa$ for $x \le \kappa$, whereas 
$\E[\nu_k \colon \mu_k \le n/g(n)] \le 0$ and 
$\frac{g(n)}{n} \nu_k \le \kappa$ when $\mu_k \le n/g(n)$, we deduce that
\begin{equation}
\E\left[\exp\left\{ \frac{g(n)}{n}\nu_k \right\}\colon\mu_k\le \frac{n}{g(n)} \right]
\le 1 + e^{\kappa} \left(\frac{g(n)}{n}\right)^2\E[\nu_k^2]
=1+o\left(\frac{g(n)}{n}\right)
\end{equation}
as $n\to\infty$ (since $g(n)/n \to 0$ and $\sup_k \E[\nu_k^2] < \infty$).
Next, $\nu_k \le \kappa \mu_k$, hence   
using integration by parts and the definition of $g(\cdot)$,
% \int_a^b e^{c x} dP_{\mu_1}(x) = e^{c a - g(a)} - e^{c b - g(b)} 
%  + c \int_a^b e^{c x - g(x)} dx 
% here:   c=g(n)\kappa/n, b = \epsilon n, a = n/g(n)
\begin{align}
&\E\left[\exp\left\{ \frac{g(n)}{n}\nu_k \right\} \colon
\frac{n}{g(n)} \le  \mu_k \le \epsilon n\right] \\
&\quad\le
\E\left[\exp\left\{ \frac{\kappa g(n)}{n}\mu_1 \right\} \colon
\frac{n}{g(n)} \le \mu_1 \le \epsilon n\right] \nonumber \\
&\quad \le  e^{\kappa-g(n/g(n))}
% -e^{\epsilon g(n) \kappa - g(\epsilon n)}
+\frac{g(n) \kappa}{n}\int_{n/g(n)}^{\epsilon n}
\exp\left\{\frac{g(n) \kappa}{n} r \right\}e^{-g(r)}\dd r \,.
\label{remainder}
\end{align}
The first term in~\eqref{remainder} is $o(g(n)/n)$ thanks to our assumption 
that $g(\cdot)$ grows faster than the logarithm.
Furthermore, from the representation formula for slowly varying 
functions~\cite[Theorem~1.2]{Sen76}, it follows that for 
all $n \ge n_0(\kappa,\epsilon)$,
\begin{equation}
\frac{g(n)}{n} \kappa r \le \frac{g(r)}{2}  \qquad 
\forall r \in[n/g(n),\epsilon n] \,. 
\end{equation}
The integral in~\eqref{remainder} is thus at most 
$\int_{n/g(n)}^\infty e^{-g(r)/2}\dd r=o(1)$ when $n\to\infty$. 
Collecting the preceding estimates, we conclude that for 
all sufficiently large $n$, 
\begin{equation}
\begin{split}
  \prod_{k=1}^n\E\left[ \exp\left\{ \frac{g(n)}{n} \nu_k \right\} \colon\mu_k\le \epsilon n\right]
&\le \left(1+\frac{\epsilon g(n)}{n}\right)^n \le \exp\set{\epsilon g(n)},
\end{split}
\end{equation}
and the proof is complete. 

\section*{Acknowledgments}
\noindent
Part of this work was done during Amir Dembo's visit to RIMS at Kyoto University, which is supported by the Kyoto University Top Global University (KTGU) Project. 
Amir Dembo was partially supported by NSF Grant Number DMS-1613091.
Ryoki Fukushima was partially supported by JSPS KAKENHI Grant Number 24740055 and 16K05200. 
Naoki Kubota was partially supported by JSPS Grant-in-Aid for Young Scientists (B) 16K17620.
The authors are grateful to the referees for the constructive comments which led to a correction to Lemma~\ref{lem:stable}.

\newcommand{\noop}[1]{}\def\cprime{$'$}

%\bibliographystyle{abbrv}
%\bibliography{DFK}

\end{document}